\documentclass[12pt]{article}
\usepackage{graphicx}
\usepackage{amssymb}
\usepackage{amsmath,amsthm}
\usepackage{verbatim}
\usepackage{amsfonts}
\usepackage[T1]{fontenc}
\usepackage{color}
\usepackage{hyperref}

\newtheorem{thm}{Theorem}[section]
\newtheorem{prop}[thm]{Proposition}
\newtheorem{cor}[thm]{Corollary}
\newtheorem{defn}[thm]{Definition}

\newtheorem{lem}[thm]{Lemma}
\newtheorem{conj}[thm]{Conjecture}

\newtheorem{notation}[thm]{Notation}
\newtheorem{rem}[thm]{Remark}

\newtheorem{example}[thm]{Example}

\newcommand{\bN}{\mathbb{N}}
\newcommand{\bR}{\mathbb{R}}

\newcommand{\Z}{\mathbb{Z}}

\renewcommand{\aa}{\alpha}

\newcommand{\vs}[1]{\langle #1 \rangle}

\title{Iterated sumsets and Hilbert functions}
\author{Shalom Eliahou\footnote{LMPA-ULCO, Calais, France. Email: eliahou(at)univ-littoral.fr}~ and 
Eshita Mazumdar\footnote{Stat-Math Unit, ISI Bengaluru. Email: eshita\_vs(at)isibang.ac.in}}

\date{}

\begin{document}
\maketitle

\begin{abstract} 
Let $A$ be a finite subset of an abelian group $(G,+)$. For $h \in \bN$, let $hA=A+\dots+A$ denote the $h$-fold iterated sumset of $A$. If $|A| \ge 2$, understanding the behavior of the sequence of cardinalities $|hA|$ is a fundamental problem in additive combinatorics. For instance, if $|hA|$ is known, what can one say about $|(h-1)A|$ and $|(h+1)A|$? The current classical answer is given by $$|(h-1)A| \ge |hA|^{(h-1)/h},$$ a consequence of Pl\"unnecke's inequality based on graph theory. We tackle here this problem with a completely new approach, namely by invoking Macaulay's classical 1927 theorem on the growth of Hilbert functions of standard graded algebras. With it, we first obtain demonstrably strong bounds on $|hA|$ as $h$ grows. Then, using a recent condensed version of Macaulay's theorem, we derive the above Plünnecke-based estimate and significantly improve it in the form
$$|(h-1)A| \ge \theta(x,h)\hspace{0.4mm}|hA|^{(h-1)/h}$$ for $h \ge 2$ and some explicit factor $\theta(x,h) > 1$, where $x \in \bR$ satisfies $x \ge h$ and $|hA|=\binom{x}{h}$. Equivalently and more simply, 
$$
|(h-1)A| \ge \frac hx\: |hA|.
$$
We show that $\theta(x,h)$ often exceeds $1.5$ and even $2$, and asymptotically tends to $e\approx 2.718$ as $x$ grows and $h$ lies in a suitable range depending on $x$.
\end{abstract}

\begin{quote}

\end{quote}

\section{Introduction}

Let $A$ be a nonempty finite subset of an abelian group $(G,+)$. For any $h \in \bN_+=\{1,2,\dots\}$, we denote by $hA$ the $h$-fold iterated sumset of $A$, i.e.
$$
hA = A+\dots+A =  \{x_1+\dots+x_h \mid x_i \in A \, \textrm{ for all } 1 \le i \le h\}.
$$
As usual, we set $0A=\{0\}$. A central problem in additive combinatorics is to understand the behavior of $|hA|$ as $h$ grows. Asymptotically, it is known that $|hA|$ is eventually polynomial in $h$. See e.g.~\cite{K1, K2,N1}. But not much is known about this polynomial and, for $h$ small, the behavior of $|hA|$ may wildly depend on the structure, or lack thereof, of $A$. For instance, if $A$ is a subset of $\Z$ such that $|A|=n$, then
$$
h(n-1)+1 \le |hA| \le \binom{n-1+h}{h},
$$
both bounds being attained in suitable cases: arithmetic progressions for the lower bound, and so-called $B_h$-sets for the upper bound. The latter is best understood by noting that this binomial coefficient counts the number of monomials of degree $h$ in $|A|$ commuting variables. See e.g.~\cite[Sections 2.1 and 4.5]{TV} or~\cite[Section 3.2]{EF}.

\smallskip
Here we address the following question. If $h \ge 2$ and $|hA|$ is known, what estimates on $|(h-1)A|$ and $|(h+1)A|$ can one derive? The classical answer, given by Pl\"unnecke's inequality and based on graph theory~\cite{P}, is as follows:
\begin{equation}\label{plunnecke}
|(h-1)A| \ge |hA|^{(h-1)/h}.
\end{equation}
See also~\cite{R1, N1,TV}. In this paper, we derive this bound with a completely new approach, and we significantly improve it along the way. Our approach relies on Macaulay's classical 1927 theorem characterizing the Hilbert functions of standard graded algebras~\cite{M}. We apply that theorem to a suitable standard graded $K$-algebra $R=R(A)=\oplus_{h \ge 0} R_h$ having the property 
$$
\dim_K R_h = |hA|
$$
for all $h \ge 0$. Using a recent condensed version of Macaulay's theorem~\cite{E}, we  improve~\eqref{plunnecke} as follows. Denote
$$
\theta(x,h) = \frac hx \binom{x}{h}^{\hspace{-1mm} 1/h}
$$
for $x \in \bR$ and $h \in \bN$. If $|A| \ge 2$, our improved bound implies
\begin{equation}\label{strengthening}
|(h-1)A| \ge \theta(x,h) |hA|^{(h-1)/h},
\end{equation}
where $x$ is the unique real number larger than $h$ such that $|hA|=\binom{x}{h}$. This ensures $\theta(x,h) > 1$. In fact, the factor $\theta(x,h)$ often exceeds $1.5$ and even $2$, as shown in Sections~\ref{sec theta ge 1.5} and~\ref{sec theta ge 2}. For instance, for $h=12$ we have $\theta(x,12)) > 2.013$ for all $x \ge 50$. This implies in turn that if $A$ satisfies $|12A| \ge 121,\!400,\!000,\!000$, then
$$|11A| \ge 2\: |12A|^{11/12}.$$
The wide occurrence of the case $\theta(x,h) \ge 2$ is described in more detail in
Section~\ref{sec theta ge 2}. Remarkably, for $x$ large enough and suitable values of $h$ depending on $x$, the factor $\theta(x,h)$ approaches $e \approx 2.718$, the basis of the natural logarithm. For instance, this occurs for all $x \ge 10^6$ at $h=3000$. See also Section~\ref{sec high}, where strong evidence suggests that $\lim_{x \to \infty} \theta(x,\lfloor x^{1/2}\rfloor) = e$.
Three general remarks are in order here.

\begin{rem} \emph{Our results are stated for finite subsets of an abelian group $G$, but they hold more generally if $G$ is a commutative semigroup, as in~\cite{N2} for instance.}
\end{rem}

\begin{rem} \emph{
Commutative algebra has already been applied to 
estimate the growth of iterated sumsets. In particular, the Hilbert polynomial of graded modules has been used to determine the asymptotic behavior of the function $h \mapsto |hA|$, and more generally of the function $(h_1,\dots,h_r) \mapsto |B+h_1A_1+\cdots+h_rA_r|$. See~\cite{K1, K2, N2, N1}. However, to the best of our knowledge, the only previous application of Macaulay's theorem to additive combinatorics is in \cite{E}, where the above-mentioned condensed version is established and applied to yield an asymptotic solution of Wilf's conjecture on numerical semigroups.}
\end{rem}

\begin{rem} \emph{
Another way of comparing $|hA|$ with $|(h-1)A|$ has been made, at least for $A \subset \Z$, by seeking to bound the difference $|hA|-|(h-1)A|$ from below rather than the quotient $|hA|/|(h-1)A|$ from above~\cite{L}. In the study of the difference $|hA|-|(h-1)A|$, a main tool is Kneser's theorem, whereas for the quotient $|hA|/|(h-1)A|$, the classical one is Pl\"unnecke's inequality, and an additional one is now Macaulay's theorem as made plain in this paper.}
\end{rem}

There is a vast literature on Plünnecke's inequality, its rich applications to additive combinatorics and its successive refinements, such as the Plünnecke-Ruzsa inequality for instance~\cite{R2}. Besides dedicated chapters in~\cite{R1, N1, TV}, see also the nice survey~\cite{Pet} and its many references.

\smallskip
The contents of this paper are as follows. In Section~\ref{sec R(A)}, we construct a graded algebra $R(A)$ whose Hilbert function exactly models the sequence $|hA|_{h \ge 0}$. In Section~\ref{hilbert}, we recall Macaulay's theorem on Hilbert functions and the recent condensed version that we shall use. We prove our main results in Section~\ref{sec main}. The first ones, Theorems~\ref{main thm 1} and~\ref{main thm 2}, are obtained by applying Macaulay's theorem and its condensed version to the algebra $R(A)$. The strength of these results is then illustrated with the specific case $|5A|=100$. Here, Plünnecke's inequality implies $|4A| \ge 40$ and $|6A| \le 251$, whereas our method yields much sharper and almost optimal bounds, namely $|4A| \ge 61$ and $|6A| \le 152$. As our next main result, Theorem~\ref{thm vs}, we derive the Pl\"unnecke-based estimate~\eqref{plunnecke} from Theorem~\ref{main thm 2} and improve it by some multiplicative factor $\theta(x,h) > 1$. The numerical behavior of that factor is studied in Section~\ref{sec theta} and shown to often exceed $1.5$ and even $2$. In Section~\ref{sec presentation}, we give a presentation of $R(A)$ by generators and relations. We conclude the paper in Section~\ref{sec conclusion} with related questions and remarks.

\section{The graded algebra $R(A)$}\label{sec R(A)}

Let $A$ be a finite subset of an abelian group. Here we associate to $A$ a standard graded algebra $R(A)$ whose Hilbert function models the sequence $|hA|$. We start by recalling some basic terminology. 
\begin{defn} A \emph{standard graded algebra} is a commutative algebra $R$ over a field $K$ endowed with a vector space decomposition $R=\oplus_{i\ge 0}\, R_i$ such that $R_0=K$, $R_iR_j \subseteq R_{i+j}$ for all $i,j \ge 0$, and which is generated as a $K$-algebra by finitely many elements in $R_1$.
\end{defn}
It follows from the definition that each $R_i$ is a finite-dimensional vector space over $K$. Moreover, as $R$ is generated by $R_1$, we have $R_iR_j=R_{i+j}$ for all $i,j \ge 0$, whence $R_i=R_1^i$, the $i$-fold iterated \emph{productset} of $R_1$.
\begin{defn} Let $R=\oplus_{i\ge 0}\, R_i$ be a standard graded algebra. The \emph{Hilbert function} of $R$ is the map $i \mapsto d_i$ associating to each $i \in \bN$ the dimension
$$
d_i \ = \ \dim_{K} R_i
$$
of $R_i$ as a vector space over $K$.
\end{defn}
Thus $d_0=1$, and $R$ is generated as a $K$-algebra by any $d_1$ linearly independent elements of $R_1$. 

\medskip
Let now $(G,+)$ be an abelian group. Consider the \emph{group algebra} $K[G]$ of $G$. Its canonical $K$-basis is the set of symbols $\{t^g \mid g \in G\}$, and its product is induced by the formula
$$t^{g_1}t^{g_2}=t^{g_1+g_2}$$
for all $g_1,g_2 \in G$. Consider now $S=K[G][Y]$, the one-variable polynomial algebra over $K[G]$. Then $S$ has for $K$-basis the set 
$$\mathcal{B}=\{t^gY^n \mid g \in G, n \in \bN\},$$ and the product of any two basis elements is given by
$$
t^{g_1}Y^{n_1} \cdot t^{g_2}Y^{n_2} = t^{g_1+g_2}Y^{n_1+n_2}
$$
for all $g_1,g_2 \in G$ and all $n_1, n_2 \in \bN$. The \emph{degree} of a basis element is defined as
$$
\deg(t^gY^n) = n
$$
for all $g \in G$ and all $n \in \bN$. This endows $S$ with the structure of a graded algebra. Thus $S=\oplus_{h \ge 0}S_h$, where $S_h$ is the $K$-vector space with basis the set $\{t^g Y^h \mid g \in G\}$.
\begin{defn} Let $A=\{a_1,\dots,a_n\}$ be a nonempty finite subset of $G$. We define $R(A)$ to be the $K$-subalgebra of $S$ spanned by the set 
$$\{t^{a_1}Y,\dots,t^{a_n}Y\}.$$ 
\end{defn}
Thus $R(A)$, being finitely generated over $K$ by elements of degree 1, is a standard graded algebra. We then have $R=\oplus_{h \ge 0} R_h$, where $R_h$ is the $K$-vector space with basis the set $\{t^bY^h \mid b \in hA\}$. It follows that 
\begin{equation}\label{model}
\dim R_h=|hA|
\end{equation}
for all $h \ge 0$, as desired.

\smallskip

For future work on $R(A)$, it is algebraically important to determine the relations between its given generators $t^{a_i}Y$. This is done in Section~\ref{sec presentation}.

\section{Macaulay's theorem} \label{hilbert}

We now turn to Macaulay's theorem~\cite{M} and a recent condensed version of it~\cite{E}. Macaulay's theorem gives a necessary and sufficient condition for a numerical function $\bN \to \bN$ to be the Hilbert function of some standard graded algebra. It rests on the so-called \emph{binomial representations} of integers. Here is some background information. 
\begin{prop} Let $a \ge i \ge 1$ be positive integers. There are unique integers $a_i > a_{i-1} > \cdots > a_1 \ge 0$ such that 
$$
a = \sum_{j=1}^i \binom{a_j}j. 
$$
\end{prop}
\begin{proof} See e.g. \cite{BH, P}.
\end{proof}

This expression is called the $i$th \emph{binomial representation of $a$}. Producing it is computationally straightforward: take for $a_i$ the largest integer such that 
$
\binom{a_i}{i} \le a,
$
and complete $\binom{a_i}{i}$ by adding to it the $(i-1)$th binomial representation of $a-\binom{a_i}{i}$. We omit trails of $0$'s, if any. For instance, for $a = 10$ and $i=3$, we abbreviate $10=\binom 53+\binom 12+\binom 01$ as simply $10=\binom 53$.

\begin{notation} Let $a \ge i \ge 1$ be positive integers. Let
$ \displaystyle
a = \sum_{j=1}^i \binom{a_j}j
$
be its $i$\emph{th} binomial representation. We denote
$ \displaystyle
a^{\vs{i}} = \sum_{j=1}^i \binom{a_j+1}{j+1}
$
and $0^{\vs{i}}=0$.
\end{notation}
Note that the defining formula of $a^{\vs{i}}$ yields the $(i+1)$th binomial representation of the integer it sums to.

\smallskip
Here is one half of Macaulay's classical result, constraining the possible Hilbert functions of standard graded algebras \cite{M}.
\begin{thm}\label{thm macaulay} Let $R = \oplus_{i \ge 0} R_i$ be a standard graded algebra over a field $K$, with Hilbert function $d_i = \dim_{K}R_i$. Then for all $i \ge 1$, we have
\begin{equation}\label{macaulay}
d_{i+1} \ \le \ d_i^{\vs{i}}.
\end{equation}
\end{thm}
Remarkably, the converse also holds in Macaulay's theorem, but we shall not need it here. That is, satisfying \eqref{macaulay} for all $i \ge 1$ \textit{characterizes} the Hilbert functions of standard graded algebras. See e.g. \cite{BH,MP,P}.

\begin{example} Consider the sequence 
$$(m_0,m_1,m_2,m_3,m_4,m_5,m_6)=(1,5,15,33,61,100,152).$$
Then $m_{i+1} \le m_i^{\vs{i}}$ for all $i=1,\dots,5$ as readily checked. Hence there exists a standard graded algebra $R=\oplus_{j \ge 0} R_{j}$ whose values of $\dim R_i$ for $i=0,\dots,6$ are exactly modeled by the sequence $(m_{0},\dots,m_6)$. For instance, one may take $R=S/J$, where $S=K[X_1,\dots,X_5]$ and $J=(X_5^3, X_4X_5^2,X_3^3X_5^2)$.
\end{example}

\subsection{A condensed version}

For our new derivation of the Plünnecke-based estimate~\eqref{plunnecke}, we shall need the following condensed version of Macaulay's theorem established in \cite{E}. For $m\in \mathbb{N}$ and $x\in \mathbb{R}$, denote as usual 
$$\binom{x}{m}= \frac{x(x-1)\cdots(x-m+1)}{m!} = \prod_{i=0}^{m-1}\frac{x-i}{m-i}.$$ 
In particular, $\displaystyle \binom {x}{0} = 1$. We shall constantly need the following observations.
\begin{lem}\label{Onto}
 Let $i\geq 1$ be an integer. Then the map $y\mapsto \binom {y}{i}$ is an increasing continuous bijection (in fact, a homeomorphism) from $[i-1, \infty)$ to $[0, \infty).$ In particular, for any real numbers $y_1,y_2 \ge i-1$, we have
 \begin{equation}\label{increasing}
 y_1 \le y_2 \iff \binom {y_1}{i} \le \binom {y_2}{i}.
 \end{equation}
 \end{lem}
 \begin{proof} A direct consequence of Rolle's theorem. See e.g.~\cite[Lemma 5.6]{E}.
 \end{proof}

\begin{lem}\label{lem binom} Let $h,d \ge 1$ be positive integers. Then there exists a unique real number $x \ge h$ such that $\displaystyle d=\binom xh$. 
\end{lem}
\begin{proof} By the above lemma, there is a unique real number $x \ge h-1$ such that $\displaystyle d=\binom xh$. Since $d \ge 1$, we have $\displaystyle \binom xh \ge \binom hh$. Hence $x \ge h$ by \eqref{increasing}.
\end{proof}
 
\medskip
Here is the condensed version of Macaulay's theorem that we shall use in the next section.

\begin{thm}\label{thm condensed}
 Let $R = \oplus_{i\ge 0} R_i$ be a standard graded algebra over the field $\mathbb{K,}$
 with Hilbert function $d_i = \dim_{\mathbb{K}}R_i $ for $i \geq 0.$ Let $h \geq 1$ be an integer. Let $x \ge h-1$ be the unique real number such that $\displaystyle d_h = \binom {x}{h}.$ Then
 $$d_{h-1} \geq \binom{x-1}{h-1} \text{ and }\, d_{h+1}\leq \binom {x+1}{h+1}.$$
  \end{thm}
\begin{proof} See~\cite{E}.
\end{proof}

\section{Main results}\label{sec main}

Let $A$ be a finite subset of an abelian group with $|A| \ge 2$. If $|hA|$ is known for some $h \ge 2$, what bounds can one derive on $|iA|$ for $i \not=h$? 

\smallskip
We first recall the classical known answer, a direct consequence of Pl\"unnecke's inequality. See e.g.~\cite[Theorem 7.5, p. 217]{N1} or~\cite[Theorem 1.2.3 with $m=1$, p. 96]{R1}. 

\begin{thm}[Plünnecke]\label{thm 7.5} Let $A$ be a nonempty finite subset of an abelian group. Let $h \ge 2$ be an integer. Then $|iA| \ge |hA|^{i/h}$ for all $1 \le i \le h$.
\end{thm}

\begin{rem}\label{equiv} \emph{
Theorem~\ref{thm 7.5} is equivalent to its main case $i=h-1$, namely:
\begin{equation}\label{case i=h-1}
|(h-1)A| \ge |hA|^{(h-1)/h}.
\end{equation}
Indeed, the general case is implied by~\eqref{case i=h-1}, as shown by induction on $h$:
\begin{gather*}
|iA| \stackrel{\eqref{case i=h-1}}{\ge} (|(i+1)A|^{i/(i+1)} 
\stackrel{\textrm{ind.hyp.}}{\ge} (|hA|^{(i+1)/h})^{i/(i+1)} = |hA|^{i/h}. 
\end{gather*}}
\end{rem}
Consequently, in the sequel, we mainly focus on comparing $|hA|$ with $|(h-1)A|$ and/or $|(h+1)A|$. In this spirit, a particular case of Plünnecke's Theorem~\ref{thm 7.5} is the estimate
\begin{equation}\label{upper bound}
|(h+1)A| \le |hA|^{(h+1)/h}
\end{equation}
for all $h \ge 1$.

\medskip

In comparison, here is our first main result, obtained by applying Macaulay's Theorem~\ref{thm macaulay} to the standard graded algebra $R(A)$ defined in Section~\ref{sec R(A)}.

\begin{thm}\label{main thm 1} Let $A$ be a nonempty finite subset of an abelian group $G$. Let $h \ge 1$ be an integer. Then
\begin{equation}\label{our upper bound}
|(h+1)A| \le |hA|^{\vs{h}}.
\end{equation}
\end{thm}
The strength of Theorem~\ref{main thm 1} is illustrated in Section~\ref{sec example},  with a concrete example showing that~\eqref{our upper bound} may be much sharper than~\eqref{upper bound}. In fact, the improvement of the former over the latter is systematic, as shown by Theorem~\ref{thm vs}, Corollary~\ref{cor vs} and Remark~\ref{rem}. See also a comment in Section~\ref{sec conclusion}.

\begin{proof} Let $R=R(A)$ be the standard graded algebra associated to $A$ as defined in Section~\ref{sec R(A)}. We have $R=\oplus_{h \ge 0}R_h$, where $R_h$ denotes the homogeneous subspace of $R$ of degree $h$. By~\eqref{model} we have 
$$
\dim R_h = |hA|
$$
for all $h \ge 0$. Hence, for $h \ge 1$, a direct application of Theorem~\ref{thm macaulay} yields the claimed upper bound~\eqref{our upper bound}.
\end{proof}

\smallskip
Let us now apply Theorem~\ref{thm condensed}, the condensed version of Macaulay's theorem. We obtain the following more flexible bounds from which we shall derive and improve~\eqref{case i=h-1}.

\begin{thm}\label{main thm 2} Let $A$ be a nonempty finite subset of an abelian group $G$. Let $h \ge 2$ be an integer and $x \ge h$ the unique real number such that $\displaystyle |hA|=\binom{x}{h}$. Then
$$
|(h-1)A| \ge \binom{x-1}{h-1} \,\, \textrm{ and }\,\, |(h+1)A| \le \binom{x+1}{h+1}. 
$$
\end{thm}

\begin{proof} As above, let $R=R(A)$ be the standard graded algebra associated to $A$ with its decomposition $R=\oplus_{h \ge 0}R_h$ into the direct sum of its homogeneous subspaces of given degree, where $\dim R_h=|hA|$ for all $h \ge 0$. The claimed bounds follow from Theorem~\ref{thm condensed} applied to $R(A)$.
\end{proof}

\begin{rem} \emph{While more handy, the upper bound in Theorem~\ref{main thm 2} is slightly weaker than in Theorem~\ref{main thm 1}. Indeed, still with $x \ge h$ such that $|hA|=\binom{x}{h}$, we have
$$\displaystyle |hA|^{\vs{h}} \le \binom{x+1}{h+1}.$$
This follows from~\cite[Theorem 5.9]{E}.}
\end{rem}

Given $|hA|$, the lower bound on $|(h-1)A|$ provided by Theorem~\ref{main thm 2} may be up to $2.71$ times better, in suitable circumstances, than the one provided in~\eqref{case i=h-1} by Plünnecke's inequality. This will be shown in Sections~\ref{sec vs} and~\ref{sec theta}. We start with a concrete example demonstrating the strength of Theorems~\ref{main thm 1} and~\ref{main thm 2}.

\subsection{An example: the case $|5A|=100$}\label{sec example}

Let $A$ be a subset of an abelian group such that $|5A|=100$. The Plünnecke-based bounds given by~\eqref{case i=h-1}, namely $|4A| \ge 100^{4/5}$ and $|6A| \le 100^{6/5}$,  yield
$$
|4A| \ge 40, \ \ |6A| \le 251.
$$
In comparison, based on the condensed version of Macaulay's theorem, Theorem~\ref{main thm 2} yields the much sharper bounds
\begin{equation}\label{ex num}
|4A| \ge 58, \quad |6A| \le 161.
\end{equation}
Indeed, let $x \ge 5$ be the unique real number such that $\binom{x}{5}=100$. Then $8.69 < x < 8.7$, as follows from $\binom{8.69}{5} \approx 99.42$ and $\binom{8.7}{5} \approx 100.2$. Hence
\begin{eqnarray*}
& & |4A| \ge \binom{x-1}{4} > \binom{7.69}{4} \approx 57.2, \\
& &  |6A| \le \binom{x+1}{6} < \binom{9.7}{6} \approx 161.99.
\end{eqnarray*}
This proves \eqref{ex num}. Theorem~\ref{main thm 1}, based on the full version of Macaulay's theorem, yields even better bounds. 
\begin{prop}\label{ex 5A 100} Let $A$ be a subset of an abelian group such that $|5A|=100$. Then
\begin{equation}\label{ex num +}
|4A| \ge 61, \quad |6A| \le 152.
\end{equation}
\end{prop}
\begin{proof} The $5$th binomial representation of $100$ is given by
$$
100 = \binom{8}{5}+\binom{7}{4}+\binom{4}{3}+\binom{3}{2}+\binom{2}{1}.
$$
The inequality $|(h+1)A| \le |hA|^{\vs{h}}$ of Theorem~\ref{main thm 1} then yields~\eqref{ex num +}. More precisely, we have
\begin{eqnarray}
|4A| & \ge & \binom{7}{4}+\binom{6}{3}+\binom{4}{2} \, = \, 61, \label{4A} \\
|6A| & \le & \binom{9}{6}+\binom{8}{5}+\binom{5}{4}+\binom{4}{3}+\binom{3}{2} \, = \, 152. \label{6A}
\end{eqnarray}
Thus inequality~\eqref{6A} directly follows from Theorem~\ref{main thm 1}. As for~\eqref{4A}, if we had 
$\displaystyle |4A| \le 60 = \binom{7}{4}+\binom{6}{3}+\binom{3}{2}+\binom{2}{1},$
then that same theorem would imply 
$\displaystyle
|5A| \le \binom{8}{5}+\binom{7}{4}+\binom{4}{3}+\binom{3}{2}=100-2$, contrary to the hypothesis.
\end{proof}

\smallskip
Is Proposition~\ref{ex 5A 100} best possible for sets satisfying $|5A|=100$? Strong evidence shows that it is not far from it. For instance, let $A = \{0,1,5,8,49\} \subset \Z$. Then $|5A| = 100$ as required, and
$$
|4A| = 63, \quad
|6A| = 145.
$$
We conjecture that these bounds are optimal for sets of integers.

\begin{conj}\label{conj 100} Let  $A \subset \Z$ be any subset satisfying $|5A|=100$. Then $|4A| \ge 63$ and $|6A| \le 145$.
\end{conj}

As seen here, the improvement provided by Theorem~\ref{main thm 2} is already quite good. How good is it in general? We investigate this question in the sequel.

\subsection{Macaulay vs Pl\"unnecke}\label{sec vs}
As our next main result, we show that Pl\"unnecke's Theorem~\ref{thm 7.5} also follows from our Macaulay-based Theorem~\ref{main thm 2}, and we significantly strengthen it by a multiplicative factor which may exceed $2.71$ in suitable circumstances.

\begin{notation}\label{notation theta} For a positive integer $h$ and a real number $x \ge h$, we set
$$
\theta(x,h) = \frac hx \binom{x}{h}^{\hspace{-1mm} 1/h}.
$$
\end{notation}

\begin{thm}\label{thm vs} Let $A$ be a nonempty finite subset of an abelian group $G$. Let $h \ge 2$ be an integer. Then
$$
|(h-1)A| \ge \theta(x,h)\:  |hA|^{(h-1)/h},
$$
where $x \in \bR$ satisfies $x \ge h$ and $\displaystyle |hA|=\binom{x}{h}$.  
\end{thm}

\begin{proof} Theorem~\ref{main thm 2} yields
\begin{equation}\label{our thm}
|(h-1)A| \ge \binom {x-1}{h-1}.
\end{equation}
Now
\begin{equation}\label{rel binom}
\binom{x-1}{h-1} = \frac hx \binom{x}{h}
\end{equation}
since
$$
\binom{x}{h} = \prod_{i=0}^{h-1} \frac{x-i}{h-i} 
 =  \frac xh \prod_{i=1}^{h-1} \frac{x-i}{h-i}
 =  \frac xh \binom{x-1}{h-1}.
$$
Hence
\begin{eqnarray*}
|(h-1)A|^h & \ge & \binom {x-1}{h-1}^{\hspace{-1mm} h} \\
& = & \left(\frac hx\right)^{\hspace{-1mm} h} \binom{x}{h}^{\hspace{-1mm} h} \\
& =& \left(\frac hx\right)^{\hspace{-1mm} h} \binom{x}{h} \binom{x}{h}^{\hspace{-1mm} h-1} \\
& = & \left(\frac hx\right)^{\hspace{-1mm} h} \binom{x}{h} |hA|^{h-1}.
\end{eqnarray*}
Therefore $|(h-1)A|^h \ge \theta(x,h)^h  |hA|^{h-1}$, as desired.
\end{proof}

\begin{cor}\label{cor vs} Theorem~\ref{main thm 2} implies Plünnecke's Theorem~\ref{thm 7.5}.
\end{cor}
\begin{proof}
By Theorem~\ref{thm vs}, we only need to show $\theta(x,h) \ge 1$, or equivalently, $\theta(x,h)^h \ge 1$. Now
\begin{equation}\label{formula theta}
\theta(x,h)^h=\left(\frac hx\right)^{\hspace{-1mm} h} \binom{x}{h}  =  \prod_{i=0}^{h-1} \frac{h(x-i)}{x(h-i)},
\end{equation}
and $h(x-i) \ge x(h-i)$  for all $0 \le i \le h-1$ since $h \le x$. 
\end{proof}

\begin{rem}\label{rem} \emph{In fact, we have $\theta(x,h) > 1$ whenever $h \ge 2$ and $|hA| \ge 2$. Inded, since $|hA|=\binom{x}{h}$ with $x \ge h$, it follows that $x > h$, whence $h(x-1)>x(h-1)$, implying in turn $\theta(x,h)^h> 1$ by \eqref{formula theta}.}
\end{rem}

\smallskip
We close this section with an equivalent formulation of Theorem~\ref{thm vs}. It provides a nice inequality between $|(h-1)A|$ and $|hA|$, yet less suited to comparison purposes with Plünnecke's inequality.

\begin{thm}\label{frac} Let $A$ be a nonempty finite subset of an abelian group $G$. Let $h \ge 2$ be an integer. Then
$$
|(h-1)A| \ge \frac hx\: |hA|,
$$
where $x \in \bR$ satisfies $x \ge h$ and $\displaystyle |hA|=\binom{x}{h}$.  
\end{thm}
\begin{proof} Directly follows from Theorem~\ref{thm vs} and the formulas 
\begin{eqnarray*}
\theta(x,h) & = & \frac hx \binom{x}{h}^{\hspace{-1mm} 1/h} \\
& = &  \frac hx \: |hA|^{(h-1)/h}.
\end{eqnarray*}
Alternatively, directly follows from Theorem~\ref{main thm 2} and formula~\eqref{rel binom}.
\end{proof}

\begin{cor} Let $A$ be a nonempty finite subset of an abelian group $G$. For all $i \ge 1$, let $x_i \in \bR$ satisfy $x_i \ge i$ and $\binom{x_i}{i}=|iA|$. Let $h \ge 2$ be an integer. For all $1 \le i \le h-1$, we have
$$
|hA| \le \left(\prod_{j=i+1}^h x_j/j\right)|iA|.
$$
\end{cor}
\begin{proof} Straightforward consequence of the above theorem.
\end{proof}

\section{Behavior of $\theta(x,h)$}\label{sec theta}

We now study the numerical behavior of the function $\theta(x,h)$. Denote $e \approx 2.718$, the basis of the natural logarithm. We show that $1 < \theta(x,h) < e$ whenever $x > h \ge 2$, and that $\theta(x,h)$ asymptotically tends to $e$ in suitable circumstances. This section is slightly more informal in nature. Numerical computations and graphics were done with \emph{Mathematica\:10}~\cite{W}.

\begin{prop} For all $h \in \bN$, $x \in \bR$ such that $x > h \ge 2$, we have
$$
1 < \theta(x,h) < e.
$$
\end{prop}
\begin{proof}
The lower bound follows from \eqref{formula theta} and Remark~\ref{rem}. As for the upper bound, we have
$$
\binom xh  \le  \frac {x^h}{h!} 
 =  \frac {x^h}{h^h}  \frac {h^h}{h!} 
 <  \frac {x^h}{h^h} e^h
$$
since $\displaystyle \frac {h^h}{h!} < \sum_{k \in \bN} \frac {h^k}{k!}=e^h$. It follows that
$$
\theta(x,h) = \frac hx \binom xh^{1/h} < \frac hx \frac xh e = e. \qedhere 
$$
\end{proof}

We shall also need to invoke the monotonicity of $\theta(x,h)$ in $x$.

\begin{prop}\label{prop increasing} For a fixed integer $h \ge 2$, the map $x \mapsto \theta(x,h)$ from $[h,\infty)$ to $[1, \infty)$ is strictly increasing.
\end{prop}
\begin{proof} It is equivalent to show that the map $x \mapsto \theta(x,h)^h$ is strictly increasing. This easily follows from the positivity of its derivative. Details are left to the reader.
\end{proof}

\subsection{Asymptotics}
We provide here, somewhat informally, a good approximation of $\theta(x,h)$ together with its asymptotic behavior as $x$ grows. Recall Stirling's approximation of $n!$ for large $n$:
$$
n! \sim \sqrt{2 \pi n}\left(\frac n e\right)^n.
$$
On the other hand, the bounds below are valid for all $n \ge 1$:

$$
\sqrt{2 \pi}\:n^{n+1/2}\:e^{-n} \le n! \le e\:n^{n+1/2}\:e^{-n}.
$$
This yields the following well known approximation of $\displaystyle \binom nk$ for $n$ much larger than $k$, see e.g.~\cite{Wi}: 
$$
\binom nk \sim \frac{(n/k-1/2)^k\:e^k}{\sqrt{2 \pi k}}.
$$

As a consequence, here is the asymptotic behavior of $\theta(x,h)$ when $x$ grows. 

\begin{prop}\label{prop limit} Let $h \ge 2$ be an integer. Then
$$
\theta(x,h) \sim \frac{(1-h/(2x))\:e}{(2 \pi h)^{1/(2h)}}=\frac{(2x-h)\:e}{2x(2 \pi h)^{1/(2h)}}.
$$
In particular, 
$$
\lim_{x \to \infty}\theta(x,h) = (2 \pi h)^{-1/(2h)}\:e.
$$
\end{prop}
\begin{proof} Directly follows from the above approximation of the binomial coefficients.
\end{proof}

\subsection{When $\theta(x,h) \ge 1.5$}\label{sec theta ge 1.5}

Our multiplicative improvement factor $\theta(x,h)$ over the Plünnecke-based estimate
\begin{equation}\label{plu}
|(h-1)A| \ge |hA|^{(h-1)/h}
\end{equation}
exceeds $1.5$ quite early in terms of $x$ or $h$. Indeed, one observes that \emph{the smallest integer $x$ for which $\theta(x,h) \ge 1.5$ for some integer $h$ is $x=10$}, specifically at $h=4$ and $5$. Even starting at $h=3$, we have 
\begin{equation}\label{case h=3}
\theta(x,3) \ge 1.509
\end{equation}
for all $x \ge 12$. As an example of application, these observations, together with Theorem~\ref{thm vs}, yield the following improvements of~\eqref{plu} for $h=3,4,5$.

\begin{cor} Let $A$ be a finite subset of an abelian group $G$. Then
\begin{eqnarray*}
|3A| \ge 220 & \Longrightarrow & |2A| \ge \frac 32 |3A|^{2/3}, \\
|4A| \ge 210  & \Longrightarrow & |3A| \ge \frac 32 |4A|^{3/4}, \\
|5A| \ge 252  & \Longrightarrow & |4A| \ge \frac 32 |5A|^{4/5}.
\end{eqnarray*}
\end{cor}
\begin{proof} For $h=3$, let $x$ be the unique real number greater than $3$ such that $|3A|=\binom{x}{3}$. Since $|3A| \ge 220=\binom{12}{3}$, we have $x \ge 12$, whence $\theta(x,3) \ge 1.5$ by~\eqref{case h=3}. The conclusion follows from Theorem~\ref{thm vs}. For $h=4$ and $5$, we have $\binom{10}{4}=210$ and $\binom{10}{5}=252$. The rest of the proof is similar, using the above-mentioned estimates $\theta(x,4), \theta(x,5)  \ge 1.5$ for all $x \ge 10$.
\end{proof}

\subsection{When $\theta(x,h) \ge 2$}\label{sec theta ge 2}
We now examine circumstances guaranteeing $\theta(x,h) \ge 2$, a case of interest where our bound in Theorem~\ref{thm vs} is at least twice better than~\eqref{plu}. As it turns out, for $x$ large, one has $\theta(x,h) \ge 2$ for almost all integers $h$ between $6$ and $\lfloor x/2 \rfloor$. We also describe cases where $\theta(x,h)$ gets very close to its upper bound $e$. 

\smallskip
So, under what minimal circumstances, in terms of $h$ or of $x$, do we have $\theta(x,h) \ge 2$? First note that if $y \ge h-1$ then $\theta(y,h) < \lim_{x \to \infty}\theta(x,h)$, as follows from Proposition~\ref{prop increasing}. Moreover, $\lim_{x \to \infty}\theta(x,h_1)$ $\le$ $\lim_{x \to \infty}\theta(x,h_2)$ whenever $h_1 \le h_2$, as  follows from Proposition~\ref{prop limit}.

That being said, consider the case $h=5$. Since $\lim_{x \to \infty}\theta(x,5) < 1.926$ by Proposition~\ref{prop limit}, the values $1 \le h \le 5$ are excluded for the occurrence of $\theta(x,h) \ge 2$.  However, already $h=6$ qualifies, as $\lim_{x \to \infty}\theta(x,6) > 2.007$. More precisely, we have
\begin{equation}\label{ge 2}
\theta(x,6) \ge 2
\end{equation}
for all $x \ge 1210$, the least integer with that property. 

If now $h$ is allowed to grow, then $\theta(x,h) \ge 2$ may occur for much smaller values of $x$. Indeed, the smallest $x \in \bN$ for which $\theta(x,h) \ge 2$ for some $h$ is $x=48$, namely at $h=11$ and $12$. More precisely, we have
$$
\begin{matrix}
\theta(48,11) > 2.001, & \theta(48,12) > 2.002, \\
\theta(48,10) < 1.997, & \theta(48,13) < 1.999.
\end{matrix}
$$
See Figure~\ref{graph 1}.

\begin{figure}[h]
\begin{center}
\includegraphics[width=9cm]{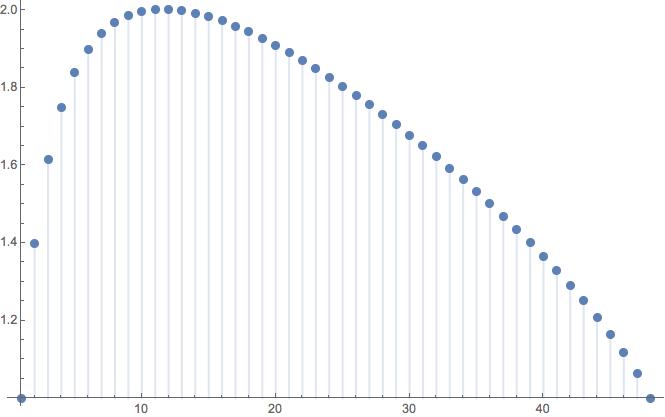}
\end{center}
\caption{Values of $\theta(48,h)$ for $h=1, \dots,48$}
\label{graph 1}
\end{figure}

In fact, when $x$ goes to infinity, then $\theta(x,h) \ge 2$ \emph{holds for almost all positive integers $h \le x/2$}. Indeed, as observed in~\eqref{ge 2}, we have $\theta(x,6) \ge 2$ for all $x \ge x_0=1210$. Now, numerical computations at $x_0$ yield 
\begin{equation}\label{1210}
\theta(x_0,h) \ge 2 \quad \forall h \in [6,x_0/2-10] \cap \bN.
\end{equation}
Together with Theorem~\ref{thm vs}, this yields the following factor 2 improvement over the Plünnecke-based estimate~\eqref{plu}.
\begin{cor} Let $h$ be an integer such that $6 \le h \le 595$. Let $A$ be a subset of an abelian group $G$ such that $|hA| \ge \binom{2h+20}{h}$. Then 
$$
|(h-1)A| \ge 2 |hA|^{(h-1)/h}.
$$
\end{cor}
\begin{proof} Let $x \ge h$ satisfy $|hA|=\binom{x}{h}$. Since $|hA| \ge \binom{2h+20}{h}$, it follows that $x \ge 2h+20 \ge 1210=x_0$. Hence $h \in [6,x_0/2-10] \cap \bN$, whence $\theta(x,h) \ge 2$ by~\eqref{1210} and Proposition~\ref{prop increasing}. The conclusion follows from Theorem~\ref{thm vs}.
\end{proof}

As yet another instance, for $x_1=10^6$ now, one has an almost identical statement as in~\eqref{1210} for $x_0=1210$, namely
\begin{equation}\label{10^6}
\theta(x_1,h) \ge 2 \quad \forall h \in [6,x_1/2-19] \cap \bN.
\end{equation}

Statements~\eqref{1210} and~\eqref{10^6} are no accident, as hinted by the following result.

\begin{prop} One has $\lim_{x \to \infty}\theta(x,\lfloor x/2\rfloor) = 2$.
\end{prop} 
\begin{proof} Using Stirling's approximation formula of $n!$, one readily sees that
$$
\theta(n,\lfloor n/2\rfloor) \approx 2 \left(\frac{2}{\pi n}\right)^{1/n},
$$
which proves the claim since $\lim_{n \to \infty}(c n)^{-1/n} = 1$ for any constant $c > 0$.
\end{proof}

\subsection{The highest point}\label{sec high}

For fixed $x$, the general shape of $\theta(x,h)$ when $h$ runs from $1$ to $\lfloor x \rfloor$ is well illustrated by Figure~\ref{graph 1} for $x=48$. Figure~\ref{graph 2} displays the case $x=1000$.

\begin{figure}[h]
\begin{center}
\includegraphics[width=9cm]{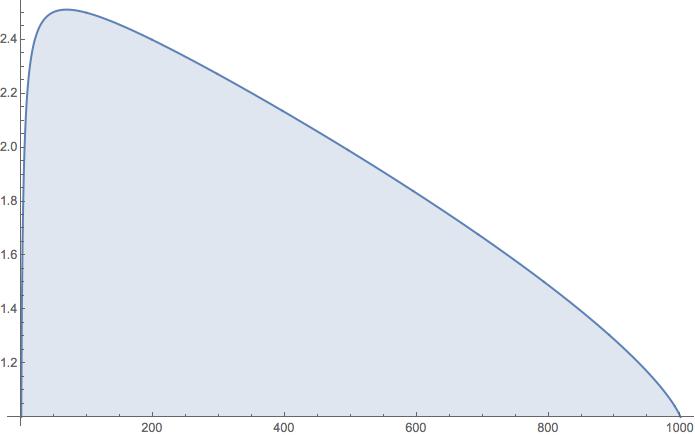}
\end{center}
\caption{Values of $\theta(1000,h)$ for $h=1, \dots,1000$}
\label{graph 2}
\end{figure}

It would be desirable to determine the highest point of that curve, and in particular the integer $1 \le h \le x$ maximizing $\theta(x,h)$. We do not have yet a precise answer. Nevertheless, by computing derivatives of the approximation of $\theta(x,h)$ provided by Proposition~\ref{prop limit}, one sees that for fixed $x$,
\begin{equation}\label{positive derivative}
\frac{\partial}{\partial h} \left(\frac{2x-h}{(2 \pi h)^{1/(2h)}}\right) > 0 \iff 2h^2 < (2x-h)(\ln(2 \pi h)-1).
\end{equation}
Thus, for $x$ fixed, the sought-for integer $h$ maximizing $\theta(x,h)$ occurs when
\begin{equation}\label{max}
2h^2 \approx (2x-h)(\ln(2 \pi h)-1).
\end{equation}
For instance, for $x_0=100$, the maximum of $\theta(x_0,h)$ is reached at $h=18$, for which $\theta(100,18) \approx 2.177$. Hence $$\theta(x,18) \ge 2.177$$ for all $x \ge 100$, as follows from Proposition~\ref{prop increasing}.

\subsection{For $h$ fixed}\label{sec h fixed}
In the opposite direction, for $h$ fixed, it is easy to locate the real number $x_1 \ge h$ maximizing $\theta(x,h)$. Indeed, using~\eqref{positive derivative}, we find 
$$
x_1 \approx \frac 12\left(\frac{2h^2}{\ln(2 \pi h)-1}+h\right).
$$
This suggests that 
$$
\lim_{x \to \infty} \theta(x,\lfloor x^{1/2}\rfloor) = e,
$$
as is fully confirmed by numerical experiments. As a concrete illustration, here are instances where $\theta(x,h)$ gets very close to $e$:
\begin{itemize}
\item For all $x \ge 200000$ and all $1200 \le h \le 1300$, one has
$\theta(x,h) \ge 2.70$.  \vspace{-0.2cm}
\item Similarly, for all $x \ge 1100000$ and all $2600 \le h \le 3700$, one has
$\theta(x,h) \ge 2.71$.
\end{itemize}

\section{A presentation of $R(A)$}\label{sec presentation}

Reusing the notation of Section~\ref{sec R(A)}, let $A=\{a_1,\dots,a_n\}$ be a nonempty finite subset of an abelian group $(G,+)$. For future use, it is algebraically necessary to determine the relations between the given generators $t^{a_i}Y$ of the associated algebra $R(A)$. Our aim here is thus to identify $R(A)$ as the quotient of the polynomial algebra $K[X_1,\dots,X_n]$ by a suitable homogeneous ideal $I$.

\begin{notation}
For $\alpha=(\aa_1,\dots,\aa_n) \in \bN^n$, let $X^\alpha = X_1^{\aa_1}\cdots X_n^{\aa_n}$ denote the corresponding monomial in $K[X_1,\dots,X_n]$. We denote the set of these monomials by $M = \{X^{\alpha} \mid \alpha \in \bN^n\}$ .
\end{notation}

Let $\varphi \colon K[X_1,\dots,X_n] \to R(A)$ be the surjective morphism induced by $\varphi(X_i)=t^{a_i}Y$ for all $i$.
On the set $M$, we define the equivalence relation
$$
u \sim v \iff \varphi(u)=\varphi(v)
$$
for all $u,v \in M$. Equivalently, let us write $u = X^\alpha, v=X^\beta$ with $\alpha=(\alpha_1,\dots, \alpha_n), \beta=(\beta_1,\dots,\beta_n) \in \bN^n$. Then 
$$
X^\alpha \sim X^\beta
\, \iff \,
\left\{
\begin{array}{lll}
\sum_i \alpha_i & = & \sum_i \beta_i, \\
\sum_i \alpha_i a_i & = & \sum_i \beta_i a_i.
\end{array}
\right.
$$
In particular, equivalent monomials have the same degree, where as usual $\deg(X^\alpha)=\sum_i \alpha_i$. 

We shall need the notion of \emph{simple polynomial} relative to $\sim$. 

\begin{defn} Let $f \in K[X_1,\dots,X_n]$. We say that $f$ is \emph{simple} if $f \not=0$ and all monomials occurring in $f$ are equivalent under $\sim$. 
\end{defn} 
Observe that a simple polynomial is homogeneous. Indeed, equivalent monomials under $\sim$ have the same degree as observed above. Moreover, every nonzero polynomial $g \in K[X_1,\dots,X_n]$ may be decomposed, in a unique way up to order, as the sum $g=f_1+\cdots+f_r$ of maximal simple polynomials $f_i$, in the sense that for all $i \not= j$, the monomials occurring in $f_i$ are non-equivalent under $\sim$ to those of $f_j$. The $f_i$ are obtained by simply regrouping the monomials of $f$ into maximal equivalence classes. We shall refer to the $f_i$ as the \emph{simple components} of $f$. See e.g. \cite[p. 232]{E0} and \cite[p. 346]{EV}, where similar notions were used.

\begin{lem}\label{lem simple} Let $g \in \ker(\varphi) \setminus \{0\}$. Then every simple component of $g$  belongs to $\ker(\varphi)$.
\end{lem}
\begin{proof} Let $f$ be a simple component of $g$. We must show $\varphi(f)=0$. Since $f$ is simple, it is homogeneous of some degree $h$. Write $f = \sum_i \lambda_i u_i$, where $\lambda_i \in K\setminus \{0\}$ for all $i$ and where the $u_i$ are pairwise distinct monomials. Since the $u_i$ are pairwise equivalent under $\sim$, we have $\varphi(u_i)=t^b Y^h$ for some $b \in hA$ independent of $i$. Hence 
$$\varphi(f)=(\sum_i \lambda_i) t^b Y^h.$$
Now, for any monomial $v$ occurring in $g$ but not in $f$, we have $\varphi(v)\not=t^b Y^h$ as $v$ is non-equivalent to the $u_i$. Since $\varphi(g)=0$, it follows that
$
\sum_i \lambda_i = 0.
$
Hence $\varphi(f)=0$, as desired.
\end{proof}

\begin{prop}  Let $I \subset K[X_1,\dots,X_n]$ be the ideal generated by the set 
$\{u-v \mid u,v \in M, u \sim v\}$. Then $\ker(\varphi)=I$.
\end{prop}
\begin{proof} We have $I \subset \ker(\varphi)$ by construction. Conversely, let $0 \not= f \in \ker(\varphi)$. By Lemma~\ref{lem simple}, we may further assume that $f$ is simple.  Write $f = \sum_{i=1}^r \lambda_i u_i$, where $\lambda_i \in K\setminus \{0\}$ for all $i$ and where the $u_i$ are pairwise distinct monomials. Since $\varphi(f)=0$ and $\varphi(u_i)=\varphi(u_j)$ for all $i \not= j$, it follows that $\sum_{i=1}^r \lambda_i=0$. Therefore $\lambda_r=-\sum_{i=1}^{r-1}\lambda_i$, and so
$$
f = \sum_{i=1}^{r-1} \lambda_i(u_i-u_r).
$$
Since $u_i \sim u_r$  for all $i$, it follows that $u_i-u_r \in I$. Hence $f \in I$, as desired.
\end{proof}

\begin{cor} We have $R(A)\simeq K[X_1,\dots,X_n]/I$.
\end{cor}
\begin{proof} By Noether's isomorphism theorem.
\end{proof}

\section{Concluding comments}\label{sec conclusion}

We end this paper with a few related questions and remarks. 

\medskip
A first natural question is, how far from optimal are our new bounds? More precisely, let $(G,+)$ be an abelian group, and let $h,i,m$ be positive integers such that $m \le |G|$. Among all subsets $A \subseteq G$ such that $|hA|=m$, what is 
\begin{itemize}
\item (inverse problem) the best possible lower bound on $|iA|$ for $i \le h$? \vspace{-0.2cm}
\item (direct problem) the best possible upper bound on $|iA|$  for $i \ge h$?
\end{itemize}
Accordingly, let us denote
$$
\omega_G(h,i,m)=
\left\{
\begin{array}{lll}
\min_{A \subseteq G} |iA| & \textrm{if} & i \le h, \\
\max_{A \subseteq G} |iA| & \textrm{if} & i \ge h,
\end{array}
\right.
$$
where in both cases, the extremum is taken over all subsets $A$ of $G$ satisfying $|hA|=m$. 

Focusing here on the direct problem with $i=h+1$, how large can $\omega_G(h,h+1,m)$ be? The upper bounds given successively by Plünnecke's inequality~\eqref{upper bound}, Theorem~\ref{main thm 2} based on the condensed version of Macaulay's theorem, and Theorem~\ref{main thm 1} based on Macaulay's theorem proper, are
\begin{eqnarray}
\omega_G(h,h+1,m) & \le & m^{(h+1)/h}, \label{omega 1} \\
\omega_G(h,h+1,m) & \le & \binom{x+1}{h+1}=\frac {x+1}{h+1} m, \label{omega 2} \\
\omega_G(h,h+1,m) & \le & m^{\vs{h}}, \label{omega 3}
\end{eqnarray}
respectively, where $x \ge h$ satisfies $\binom xh=m$. Applied to the case $(h,m)=(5,100)$ in Section~\ref{sec example}, these bounds yield successively
$$
\omega_G(5,6,100) \le
\left\{
\begin{array}{rcl}
251 & \textrm{by} & \eqref{omega 1}, \\
161 & \textrm{by} & \eqref{omega 2}, \\
152 & \textrm{by} & \eqref{omega 3}.
\end{array}
\right.
$$
The last one is probably close to optimal. Indeed, for $G=\Z$, we gave an example with $|5A|=100$ and $|6A|=145$, yielding $\omega_\Z(5,6,100) \ge 145$. Conjecture~\ref{conj 100} implies that this is best possible, i.e. that $\omega_\Z(5,6,100) = 145$. 

\medskip
As another natural question, can one specialize Macaulay's theorem by characterizing the Hilbert functions of all algebras of the form $R(A)$ for finite subsets $A$ of a given abelian group $G$? A positive answer would help tackle the former question.

\medskip
Finally, in a sequel to this paper, we will show two more aspects of the strength of Theorem~\ref{main thm 1}. The proof methods are quite different from the present ones, except that Macaulay's theorem remains central. First, we will show that Theorem~\ref{main thm 1} is \emph{asymptotically optimal}: the upper bound it provides, namely
$$
|(h+1)A| \le |hA|^{\vs{h}}
$$
for all $h \ge 1$, \emph{is in fact an equality for $h$ large enough}. Second, we will show that Theorem~\ref{main thm 1} is \emph{best possible} in the sense that, given any sequence of positive integers $(d_i)_{i \ge 0}$ such that $d_0=1$ and
$$
1 \le d_{i+1} \le d_i^{\vs{i}}
$$
for all $i \ge 1$, there exists a finite subset $A$ of a commutative semigroup $(G,+)$ such that
$$
d_h=|hA|
$$
for all $h \ge 0$. 

\medskip
Together, the present paper and its forthcoming sequel raise the prospect that Macaulay's theorem, an almost century-old classical result from commutative algebra, may emerge as a powerful new tool in additive combinatorics.

\bigskip
\noindent
\textbf{Acknowledgments.} This research was supported in part by the International Centre for Theoretical Sciences (ICTS) during a visit for the program - Workshop on Additive Combinatorics (Code: ICTS/wac2020/02). We are grateful to David Grynkiewicz for very useful discussions concerning this work during the ICTS Workshop.


\medskip
\noindent
\textbf{Authors' addresses:}

\begin{itemize}
\item[$\triangleright$] Shalom Eliahou, \\
Univ. Littoral C\^ote d'Opale, UR 2597 - LMPA - Laboratoire de Math\'ematiques Pures et Appliqu\'ees Joseph Liouville, F-62228 Calais, France and CNRS, FR2037, France. \\
\texttt{eliahou@univ-littoral.fr}
\item[$\triangleright$] Eshita Mazumdar, \\
Stat-Math Unit, ISI Bengaluru. \\
\texttt{eshita\_vs(at)isibang.ac.in}
\end{itemize}

\end{document}